\documentclass[reqno]{amsart}
\usepackage{fixltx2e}
\usepackage{amssymb,amsmath,amsfonts,amsthm}
\usepackage[utf8]{inputenc}
\usepackage[T1]{fontenc}
\usepackage[all]{xy}
\usepackage{mathbbol}
\usepackage{mathabx}
\usepackage{graphicx}
\usepackage{xcolor}

\def\t{\otimes}

\newcommand{\defn}{\textbf}

\DeclareMathOperator{\Hom}{Hom}
\DeclareMathOperator{\Tor}{Tor}
\DeclareMathOperator{\Ext}{Ext}

\DeclareMathOperator{\J}{J}
\DeclareMathOperator{\U}{U}
\DeclareMathOperator{\FF}{\mathit{F}}
\DeclareMathOperator{\HH}{H}
\DeclareMathOperator{\Om}{\Omega}
\DeclareMathOperator{\BLb}{BLb}
\DeclareMathOperator{\ad}{ad}
\DeclareMathOperator{\Innder}{InnDer}

\DeclareMathOperator{\lie}{Lie}
\DeclareMathOperator{\Span}{Span}
\DeclareMathOperator{\Der}{Der}

\newcommand{\cc}{\mathbb{C}}
\newcommand{\kk}{\mathbb{K}}
\newcommand{\Ab}{\mathsf{Ab}}
\newcommand{\Mod}{\mathsf{Mod}}
\newcommand{\Lie}{\mathsf{Lie}_{\kk}}
\newcommand{\Ass}{\mathsf{Alg}_{\kk}}
\newcommand{\LMod}{\text{$L$-}\mathsf{Mod}_{\kk}}
\newcommand{\nLie}{\text{$n$-}\mathsf{Lie}_{\kk}}
\newcommand{\nLeib}{\text{$n$-}\mathsf{Leib}_{\kk}}
\newcommand{\tensor}{\otimes}

\newcommand{\noproof}{\hfil\qed}

\newtheorem{Th}[subsection]{Theorem}
\newtheorem{Pro}[subsection]{Proposition}
\newtheorem{Le}[subsection]{Lemma}
\newtheorem{Co}[subsection]{Corollary}
\theoremstyle{definition}
\newtheorem{De}[subsection]{Definition}
\newtheorem{Ex}[subsection]{Example}

\theoremstyle{remark}
\newtheorem{Rem}[subsection]{Remark}

\begin{document}

\title[Do $n$-Lie algebras have universal enveloping algebras?]{Do $n$-Lie algebras have\\ universal enveloping algebras?}
\author{X. García-Martínez}
\address{[X. García-Martínez] Department of Algebra, University of Santiago de Compostela, 15782 Santiago de Compostela, Spain.}
\email{xabier.garcia@usc.es}
\author{R. Turdibaev}
\address{[R. Turdibaev] Department of Algebra, University of Santiago de Compostela, 15782 Santiago de Compostela, Spain.}
\email{rustamtm@yahoo.com}
\author{T. Van der Linden}
\address{[T. Van der Linden] Institut Recherche en Mathématique et Physique, Université catholique de
Louvain, chemin du cyclotron 2 bte L7.01.02, 1348 Louvain-la-Neuve, Belgium}
\email{tim.vanderlinden@uclouvain.be}

\thanks{
The authors were supported by Ministerio de Economía y 
Competitividad (Spain), grant MTM2013-43687-P (European 
FEDER support included).
The first two authors were also supported by Xunta de Galicia, 
grant GRC2013-045 (European FEDER support included), and the first author 
by an FPU scholarship, Ministerio de Educación, Cultura y Deporte (Spain).
The third author is a Research Associate of the Fonds de la Recherche Scientifique--FNRS}

\begin{abstract}
The aim of this paper is to investigate in which sense, for $n\geq 3$, $n$-Lie algebras admit universal enveloping algebras.
There have been some attempts at a construction (see~\cite{Dzh} and~\cite{Bava}) but after analysing those we come to the conclusion that they cannot be valid in general. 
We give counterexamples and sufficient conditions. 

We then study the problem in its full generality, showing that universality is incompatible with the wish that the category of modules over a given $n$-Lie algebra $L$ is equivalent to the category of modules over the associated algebra $\U(L)$. Indeed, an \emph{associated algebra functor} $\U \colon \nLie \to \Ass$ inducing such an equivalence does exist, but this kind of functor never admits a right adjoint. 

We end the paper by introducing a (co)homology theory based on the associated algebra functor $\U$.
\end{abstract}
\subjclass[2010]{}
\keywords{}

\maketitle

\section{Introduction}

The algebraic concept of an \emph{$n$-Lie algebra} (also called a \emph{Filippov algebra} or a \emph{Nambu algebra}) 
is a natural generalisation of Lie algebras. Alternative generalisations of Lie algebras to $n$-ary brackets exist, 
such as \emph{Lie triple systems}~\cite{Jac}, but we shall not study those in the present paper.
By definition, an $n$-Lie algebra is a $\kk$-module with a skew-symmetric $n$-ary operation which is also a derivation. In recent years these have shown their relevance
in some areas of physics such as Nambu mechanics~\cite{Nam} or string and membrane theory~\cite{BaLa1,BaLa2}.

In this article we investigate how to extend the concept of \emph{universal enveloping algebra}, an important basic tool in theory of ordinary (= $2$-) Lie algebras, to $n$-Lie algebras where $n\geq 3$.

Given a Lie algebra $L$, its universal enveloping algebra $\U(L)$ has three distinguishing characteristics:
\begin{enumerate}
\item[(U1)] \emph{equivalent representations}: the category of Lie modules over $L$ is equivalent to the category of ``standard'' modules over $\U(L)$;
\item[(U2)] \emph{universality}: the functor $\U \colon \Lie \to \Ass$ has a right adjoint
\[
(-)_{\lie} \colon {\Ass \to \Lie}
\]
which endows an associative algebra with a Lie algebra structure via the bracket
$[a, b] = ab - ba$;
\item[(U3)] enveloping algebras are \emph{enveloping}: if $L$ is free as $\kk$-module (for instance, whenever $\kk$ is a field), then the $L$-component $\eta_{L}\colon {L\to \U(L)}$ of the unit~$\eta$ of the adjunction considered in (U2) is a monomorphism~\cite{Hig}.
\end{enumerate}

In the literature, already some attempts at introducing universal enveloping algebras for $n$-Lie algebras have been made~\cite{Bava,Dzh}. However, in the beginning of Section~\ref{S:Constr} we give an example showing that those cannot be fully valid. 
The problem with these approaches is that they depend on the existence of a functor from $\nLie$ to $\Lie$, analogous to the Daletskii-Takhtajan functor
for Leibniz algebras~\cite{DaTa}. The construction proposed in~\cite{Bava}, though, produces an object is not always a Lie algebra. That is to say, the ``functor'' in question does not land in the right category. In Corollary~\ref{Abelian} and Proposition~\ref{2dim}
we give some conditions which establish when the construction of~\cite{Bava} is, or isn't, a Lie algebra. Luckily, this imprecise definition is not an obstruction to further results in the papers~\cite{Bava,Dzh}, since those focus on simple $n$-Lie algebras over the complex numbers, and in Remark~\ref{notnec} we explain that for those $n$-Lie algebras the construction proposed in~\cite{Bava,Dzh} does indeed work. Nevertheless, the general definition of universal enveloping algebra proposed there is not correct.

One problem we face when extending the concept of universal enveloping algebra to the category of $n$-Lie algebras is the lack of a natural generalisation of the functor $(-)_{\lie}$, so that $\U$ cannot be defined via (U2). 
Therefore, using a standard categorical technique, in Section~\ref{S:Univ} we define a functor $\U \colon \nLie \to $ $\Ass$ such that (U1) holds: 
the category of modules over an $n$-Lie algebra $L$ is equivalent to the category of $\U(L)$-modules. It happens that this functor does not have a right adjoint. In fact, we prove that \emph{any} functor
satisfying (U1) cannot have a right adjoint of the kind needed for (U2), so that the requirements (U1) and (U2) are shown to be mutually incompatible. And without condition (U2), the third requirement (U3), which asks that components of the unit of the adjunction from (U2) are monomorphisms, loses its sense. We thus end up with a functor $\U \colon \nLie \to \Ass$ satisfying just (U1), which we call the \emph{associated algebra functor}.

In the final Section~\ref{S:Hom} we extend Lie algebra (co)homology to a (co)homology theory based on this associated algebra functor
and we prove it to be different from the cohomology theories introduced in~\cite{Tak},~\cite{DaTa} and~\cite{AKMS}.

\section{Preliminaries on $n$-Lie algebras}\label{S:prel}

Let $\kk$ be a commutative unital ring and $n$ a natural number, $n \geq 2$. The following definitions first appeared in~\cite{Fil, CaLoPi}.

\subsection{$n$-Leibniz and $n$-Lie algebras}
An \defn{$n$-Leibniz algebra} $L$ is a $\kk$-module equipped with an $n$-linear 
operation $L^{n}\to L$, so a linear map $[-, \dots, -] \colon L^{\otimes n} \to L$, satisfying the identity
\begin{equation}\label{fundamental}\tag{$\star$}
\big[[x_1, \dots, x_{n}], y_1, \dots, y_{n-1}\big] = \sum\limits_{i=1}^n \big[x_1, \dots, x_{i-1}, [x_i, y_1, \dots, y_n], x_{i+1}, \dots, x_n \big]
\end{equation}
for all $x_i$, $y_i \in L$. A homomorphism of $n$-Leibniz algebras is a $\kk$-module homomorphism preserving this bracket; this defines the category $\nLeib$. 

An \defn{$n$-Lie algebra} $L$ is an $n$-Leibniz algebra of which the bracket $[-, \dots, -]$ factors through the exterior product to a morphism
\[
\Lambda^{n}L=\underbrace{L\wedge \cdots \wedge L}_{\text{$n$ factors}} \to L.
\]
We thus obtain the full subcategory $\nLie$ of $\nLeib$ determined by the $n$-Lie algebras.

The latter condition means that the bracket $[-, \dots, -]$ is not just $n$-linear, but also \defn{alternating}: it vanishes on any $n$-tuple with a pair of equal coordinates. In other words, $[x_{1},\dots,x_{n}]=0$ as soon as there exist $1\leq i<j\leq n$ for which $x_{i}=x_{j}$.

When $n=2$, identity \eqref{fundamental} yields the Leibniz identity. In this case, being alternating is equivalent to skew-symmetry, which gives the Jacobi identity. Thus the above definition describes Leibniz and Lie algebras, respectively. 

\subsection{Derivations}
A linear endomap $d\colon L\to L$ on an $n$-Lie algebra $L$ is called a 
\defn{derivation} if 
\[
d([x_1,x_2,\dots , x_n])= \sum_{i=1}^n [x_1,\dots, d(x_i),\dots , x_n].
\]

The $\kk$-module of all derivations of a given $n$-Lie algebra $L$ is 
denoted by $\Der(L)$ and forms a Lie algebra with respect to the commutator $[d_{1},d_{2}]=d_{1}d_{2}-d_{2}d_{1}$.

\subsection{Ideals}
An \defn{ideal} of an $n$-Lie algebra is a normal subalgebra. It is easily seen that a $\kk$-submodule $I$ of an $n$-Lie algebra $L$ is an ideal if and only if $[I,L,\dots, L]\subseteq I$. 

\subsection{Right multiplication, adjoint action}
Given a generator $x=x_1\otimes \dots \otimes x_{n-1}$ of $L^{\otimes(n-1)}$, the \defn{right multiplication} and the \defn{adjoint action} (also called \defn{left multiplication}) by $x$ are maps
\[
R_{x}=R(x_1,\dots,x_{n-1})\quad\text{and}\quad\ad_{x}=\ad(x_1,\dots,x_{n-1})\colon L \to L
\]
respectively defined by 
\[
R(x_1,\dots,x_{n-1})(a)=[a,x_1,\dots,x_{n-1}]
\]
and
\[
\ad(x_1,\dots,x_{n-1})(b)=[x_1,\dots,x_{n-1},b]
\]
for $a$, $b\in L$. Clearly, $\ad_x=(-1)^{n-1}R_x$, and due to identity~\eqref{fundamental} both maps are derivations. They are called \defn{inner derivations} of $L$ 
and generate an ideal $\Innder(L)$ of $\Der(L)$. We will use the same notations $\ad_x$ and $R_x$ for the extensions (by derivation) of these maps to the entire tensor algebra $T(L)$. (That is to say,  $R_{x}(a_{1}\t a_{2})=R_{x}(a_{1})\t a_{2}+a_{1}\t R_{x}(a_{2})$, etc.)

\subsection{The centre}
The ideal $Z(L)=\{ z\in L \mid \ad_x(z)=0, \forall x\in L^{\otimes(n-1)}\}$ is called the \defn{centre} of $L$. 

\subsection{Simple $n$-Lie algebras}
Given an ideal $I$, we write $I^1=[I,L,\dots,L]$ for the $\kk$-submodule 
spanned by the elements $R_x(i)$ where $i\in I$ and $x\in L^{\otimes(n-1)}$. It is easy to see that $I^1$ is an ideal of $L$. If $L^{1}\neq 0$ (so that it is non-abelian, i.e., it doesn't come equipped with the zero bracket) and $L$ does not admit any non-trivial ideals then $L$ is called a \defn{simple} $n$-Lie algebra.

We now recall from~\cite{Fil} an important example of an $(n+1)$-dimensional 
$n$-Lie algebra which is an analogue of the three-dimensional 
Lie algebra with the cross product as multiplication. 

\begin{Ex}\label{simple}
Let $\kk$ be a field and $V_n$ an $(n+1)$-dimensional $\kk$-vector space with a basis $\{ e_1,\dots, e_{n+1}\}$. 
Then $V_n$, equipped with the skew-symmetric $n$-ary multiplication induced by
\[
[e_1,\dots, e_{i-1},e_{i+1},\dots, e_{n+1}]=(-1) ^{n+1+i}e_i,\qquad 1\leq i \leq n+1,
\]
is an $n$-Lie algebra.
\end{Ex}

This algebra is a simple $n$-Lie algebra. Conversely, as shown in~\cite{Lin}, 
over an algebraically closed field $\kk$ all simple $n$-Lie 
algebras are isomorphic to $V_n$.

\subsection{Leibniz and Lie algebras associated to an $n$-Lie algebra}
Given an $n$-Lie algebra $L$, we introduce the operations
\begin{align*}
[-,-]&\colon L^{\otimes(2n-2)} \to L^{\otimes(n-1)}\colon x\tensor y\mapsto [x,y]=\ad_x(y),\\
-\circ-&\colon L^{\otimes(2n-2)} \to L^{\otimes(n-1)}\colon x\tensor y\mapsto x\circ y =\tfrac12\bigl(\ad_x(y)-\ad_y(x)\bigr).
\end{align*}
Note that $\circ$ is skew-symmetric. Furthermore, the operations coincide 
if and only if $\ad_x(y)=-\ad_y(x)$, i.e., when $[-,-]$ is skew-symmetric. 
These two products have the following property relating them to the adjoint action.

\begin{Pro}\label{commutator}
For any $x$, $y \in L^{\otimes(n-1)}$ the equality
\[
[\ad_x, \ad_y]=\ad_{[x,y]}=\ad_{x\circ y}
\]
holds.
\end{Pro}
\begin{proof}
Let $x=x_2\otimes\dots\otimes x_n$, $y=y_2\otimes \dots \otimes y_n$ and $x_1 \in L$. Then from identity (\ref{fundamental}) we deduce
\[
\ad_y\ad_x(x_1)=\ad_x\ad_y(x_1)+(-1)^{n-1}\sum_{k=2}^n[x_1,\dots, \ad_y(x_k),\dots, x_n],
\]
which is equivalent to 
\[
[\ad_x,\ad_y](x_1)=(-1)^{n}\sum_{k=2}^n[x_1,x_2,\dots, \ad_y(x_k),\dots, x_n]=-\ad_z(x_1),
\]
where $z=\ad_y(x)$.	By symmetry, $[\ad_y,\ad_x]=-\ad_w$, where $w= \ad_x(y)$.
Since $\Innder(L)$ is a Lie algebra we obtain $[\ad_x, \ad_y]=\ad_w=\ad_{[x,y]}$ 
and $[\ad_x, \ad_y]=\tfrac12 (\ad_w-\ad_v)=\ad_{x\circ y}$. 
\end{proof}

The following result is due to Daletskii and Takhtajan~\cite{DaTa}.

\begin{Th}
Let $L$ be an $n$-Lie algebra. Then $L^{\otimes(n-1)}$ with bracket $[-,-]$ is a Leibniz algebra.\noproof
\end{Th}

This algebra is called the \defn{basic Leibniz algebra} associated to an 
$n$-Lie algebra~$L$. We denote it by $\BLb_{n-1}^{\otimes} (L)$. 

Following~\cite{DaTa} let us write $\mathcal{K}_{n-1}=\Span\{ x \in L^{\otimes(n-1)} \mid \ad_x=0 \}$ for the kernel of the adjoint action. We recall the following result from~\cite{DaTa}.

\begin{Th}\label{TahThm1}
The subspace $\mathcal{K}_{n-1}$ is an ideal of $\BLb_{n-1}^{\otimes}(L)$ 
and the quotient algebra $\BLb_{n-1}^{\otimes}(L)/\mathcal{K}_{n-1}$ is a Lie algebra.\noproof
\end{Th} 

This Lie algebra was introduced in~\cite{DaKu} and called the \defn{basic Lie algebra} of the given $n$-Lie algebra $L$.

\section{Algebras associated to an $n$-Lie algebra}\label{S:Constr}

Given an $n$-Lie algebra $L$ over the complex numbers $\cc$, 
in the article~\cite{Bava} the authors consider the algebra $(\Lambda^{n-1}L, \circ)$.
In Proposition~1 of~\cite{Bava}, this product $\circ$ is claimed to satisfy the Jacobi identity. 
However, this cannot be correct, as we may see in the following example of a $3$-Lie algebra, 
which is a member of the class of so-called \emph{filiform} $3$-Lie algebras given in~\cite{GGR}. 
For the sake of simplicity let us denote 
$\J(a,b,c) = a\circ (b\circ c) + c\circ(a\circ b) + b\circ( c\circ a)$. 

\begin{Ex}
Consider the $3$-Lie algebra with basis $\{x_1,x_2,x_3,x_4,x_5\}$ and the table of multiplication determined by
\begin{align*}
[x_1 , x_2 ,x_ 3 ] = x_4, \qquad
[x_1 ,x_2 ,x_4 ] = [x_1 ,x_3 ,x_4 ] = [x_2 ,x_3 ,x_4 ] = x_5.
\end{align*}
Then $\J(x_1\wedge x_4, x_1\wedge x_2, x_3\wedge x_2)=-\tfrac14 x_4\wedge x_5 \neq 0$.
\end{Ex} 

In order to determine when the algebra $(\Lambda^{n-1}L, \circ)$ defined in~\cite{Bava} is actually a Lie algebra, let us have a look at the terms of Jacobi identity.

\begin{Pro}
Let $L$ be an $n$-Lie algebra. Then for any $a$, $b$, $c\in \Lambda^{n-1}L$ the following equality holds:
\begin{equation}\label{jacobi}\tag{$\dagger$}
\J(a,b,c)=-\tfrac14\bigl( [\ad_b,\ad_c](a)+[\ad_a,\ad_b](c)+[\ad_c,\ad_a](b) \bigr)
\end{equation}	
\end{Pro} 

\begin{proof}
By Proposition~\ref{commutator} we obtain
\begin{align*}
a\circ \bigl(b\circ c)=& \tfrac12(\ad_a(b\circ c) -\ad_{b\circ c}(a)\bigr) \\
=&\tfrac12\left(\ad_a(\tfrac12\bigl(\ad_b (c) -\ad_c (b)\bigr)\bigr)\right)-\tfrac12\ad_{b\circ c}(a)\\
=& \tfrac14 \ad_a\bigl( \ad_b(c)\bigr)-\tfrac14\ad_a\bigl( \ad_c(b)\bigr)-\tfrac12[\ad_b,\ad_c](a).
\end{align*}
After similar calculations for the other terms, equality~\eqref{jacobi} follows.
\end{proof}

\begin{Co}\label{Abelian} 
Let $L$ be an $n$-Lie algebra with abelian $\Innder(L)$. Then $(\Lambda^{n-1}L, \circ)$ is a Lie algebra.\noproof
\end{Co}

\begin{Rem}
Corollary~\ref{Abelian} provides us with a sufficient condition. Due to the results in~\cite{PoSa}, for $n=3$ it also seems to be necessary. Indeed, while considering a more general question, in that work a similar product appears. Now given a free 
skew-symmetric ternary algebra $(F,[-,-,-])$, the authors of~\cite{PoSa} consider $F\wedge F$ 
equipped with the product $x\cdot y=\ad_y(x)-\ad_x(y)=-2(x\circ y)$. Observe that 
$(F\wedge F,\cdot)$ is a Lie algebra if and only if $(F\wedge F, \circ)$ is a Lie algebra. 
\end{Rem}

\begin{Rem} \label{Pojidaev}
	It is claimed in \cite[Theorem 3.1]{PoSa} that if $I$ is a non-zero minimal ideal of $F$ such that quotient $F/I \wedge F/I$
	is a Lie algebra then
	\[
	I=\langle [[x_1,x_2,x_3],x_4,x_5]-[[x_1,x_4,x_5],x_2,x_3]\mid x_{i}\in F\rangle.
	\]
	However, this result cannot be correct. Indeed, consider the central extension $F=\mathbb{C}z\oplus V_3$ of the simple 3-Lie algebra $V_3$ over $\mathbb{C}$ from Example~\ref{simple}. We have $I=\langle [\ad_x,\ad_y](a) \mid \text{$a\in F$, $x$, $y\in F\wedge F$} \rangle \subseteq V_3$ so that $I=V_3$ because $V_{3}$ is simple. Set $a=e_2\wedge e_4$, $ b=e_1\wedge e_2$, $c=e_1\wedge z$ and observe that $\ad_c=0$, $[\ad_a,\ad_b](e_1)=e_4$, which yields $\J(a,b,c)= z \wedge e_4\neq 0$. Hence $(F\wedge F,\cdot)$ is not a Lie algebra. However, by \cite[Corollary 1.2.4]{Bal}, $(V_3\wedge V_3, \circ)$ is indeed a Lie algebra (isomorphic to $ \mathfrak{so}_4$). As a consequence, $Z(F)=\mathbb{C}z$ is another minimal ideal with the property that the quotient algebra is a Lie algebra.
\end{Rem}	

It follows that the condition of Corollary~\ref{Abelian} is sufficient, but not necessary. 
A~precise characterisation seems hard to find, but we have the following partial result.

\begin{Pro}\label{2dim}
	Let $\kk$ be a field and let $L$ be a $3$-Lie algebra over $\kk$ such that $\Innder(L)$ is not abelian. If $\dim Z(L)\geq 2$ then $(\Lambda^{n-1} L, \circ)$ is not a Lie algebra.
\end{Pro}

\begin{proof} 
	By assumption there are some $x$, $y \in L\wedge L$ with 
$[\ad_x,\ad_y]=\ad_{[x,y]}\neq 0$. Pick an element $z_1\in Z(L)$ and, if possible, take $z_2 \in L$ such that $\ad_{[x,y]}(z_2) \notin \Span \{ z_1 \}$. In other words, $ z_1\wedge\ad_{[x,y]}(z_2)\neq 0$. 
Putting $z=z_1\wedge z_2$ yields $\ad_z(L)=0$ and 
$\ad_{[y,z]}(x)+\ad_{[z,x]}(y) =0$. 
However, $[\ad_x,\ad_y](z) = z_1\wedge\ad_{[x,y]}(z_2) \neq 0$
and thus the Jacobi identity does not hold.

	If such a $z_2$ does not exist, then let us assume that $\ad_{[x,y]}(L) =\Span\{ z_1 \}$.
	In this case, pick $z_3\in Z(L)$ linearly independent from $z_1$. 
	Choose a $z_4\in L$ such that $\ad_{[x,y]}(z_4)=z_1$ and consider $z=z_3\wedge z_4$. 
	Obviously, $z\neq 0$ and $-4\J(x,y,z)= z_3\wedge z_1 $ which is not zero.
\end{proof}

In the remaining cases it is not clear whether $(\Lambda^{n-1}L, \circ)$ is a Lie algebra or not.

\subsection{The basic Leibniz algebra $\BLb_{n-1}^{\Lambda}(L)$}
It is hard to endow $\Lambda^{n-1}L$ with a Lie algebra structure but 
it inherits a Leibniz algebra structure from the basic Leibniz algebra of~\cite{DaTa}.
Consider the subspace
\[
\mathcal{W}_{n-1} = \Span\{x_1\otimes \dots \otimes x_{n-1} \mid \text{$x_i=x_j$ for some $1\leq i <j \leq n-1$}\}
\]
of $L^{\otimes(n-1)}$. 

\begin{Pro}\label{extleib} 
Let $L$ be an $n$-Lie algebra. Then $\mathcal{W}_{n-1}$ is an ideal of $\BLb_{n-1}^{\otimes}(L)$ and $\Lambda^{n-1}L = \BLb_{n-1}^{\otimes}(L) / \mathcal{W}_{n-1}$.
\end{Pro}

\begin{proof}
First, note that for any $w\in \mathcal{W}_{n-1}$ and $v \in L^{\otimes(n-1)}$ we have $[w,v]=\ad_w(v)=0$, so that $\mathcal{W}_{n-1}\subseteq \mathcal{K}_{n-1}$. 
For $w=x_1\otimes \dots \otimes x_{n-1}\in \mathcal{W}_{n-1}$, where $x_i=x_j$ 
for some $1\leq i <j\leq n-1$ and $v \in L^{\otimes(n-1)}$ we have
\begin{align*}
[v,w]=\ad_v(w)= & [x_1,\dots, \ad_v(x_i),\dots, x_j, \dots, x_{n-1}] \\ 
{} + & [x_1,\dots, x_i,\dots, \ad_v(x_j), \dots, x_{n-1}]\\
 {} + & \sum_{k\neq i, k \neq j} [x_1,\dots, \ad_v(x_k),\dots, x_{n-1}] \in \mathcal{W}_{n-1}
\end{align*} 
since the sum of the first two terms and every summand in the sum belongs to $\mathcal{W}_{n-1}$.
Hence $\mathcal{W}_{n-1}$ is an ideal of the Leibniz algebra $\BLb_{n-1}^{\otimes}(L)$ and 
we may conclude that $\Lambda^{n-1}L= \BLb_{n-1}^{\otimes}(L)/\mathcal{W}_{n-1}$.
\end{proof}

Let us denote this Leibniz algebra 
$\left(\Lambda^{n-1}L,[-,-]\right)$ by $\BLb_{n-1}^{\Lambda}(L)$. 
The basic Lie algebra $\BLb_{n-1}^{\otimes}(L)/\mathcal{K}_{n-1}$ is 
a subalgebra of the Leibniz algebra $\BLb_{n-1}^{\Lambda}(L)$. 

\begin{Rem}\label{Rem1}
In~\cite{Dzh}, given an $n$-Lie algebra $L$, the vector space $\Lambda^{n-1}L$ 
is equipped with a product $[x,y]=R_x(y)=(-1)^{n-1} \ad_x(y)$. This algebra coincides with $\BLb_{n-1}^{\Lambda}(L)$ up to a sign $(-1)^{n-1}$ in the multiplications.
This product is not skew-symmetric as shown in~\cite[Remark 1.1.16]{Bal}. 
\end{Rem}

\begin{Pro}\label{inner}
$\BLb_{n-1}^{\Lambda}(L)\cong \Innder(L)$ if and only if $\mathcal{K}_{n-1} = \mathcal{W}_{n-1}$.
\end{Pro} 
\begin{proof}
Due to Proposition~\ref{commutator} and the $n$-Lie structure of $L$, 
the map
\[
x=x_2\wedge\dots\wedge x_n \mapsto \ad_x
\]
is a well-defined 
surjective Leibniz algebra homomorphism of $\BLb_{n-1}^{\Lambda}(L)$ onto 
$\Innder(L)$. Now if the kernel of this map is zero, which means 
$\ad_{x}\neq 0$ for all $x\neq 0$, then this map is an isomorphism. 
\end{proof} 

\begin{Rem}\label{notnec} 
Consider the simple $n$-Lie algebra $V_n$ over $\cc$ of Example~\ref{simple}. 
It is easily seen that $\mathcal{K}_{n-1}=\mathcal{W}_{n-1}$, so we have an isomorphism 
\[
\BLb_{n-1}^{\otimes}(V_n)/ \mathcal{K}_{n-1} =\BLb_{n-1}^{\Lambda}(V_n)\cong\Innder(V_n).
\]
Moreover, by skew-symmetry of the bracket we have $x\circ y=[x,y]$ and therefore 
$(\Lambda^{n-1}L, \circ)$ is the same Lie algebra $\Innder(V_n)$. A different construction 
of the simple $n$-Lie algebra is given in~\cite{Bava} and it is proven 
\cite[Corollary 1.2.4]{Bal} that its basic Lie algebra happens to be 
$\mathfrak{so}_{n+1}$. Hence, the algebras constructed in~\cite{Dzh} and~\cite{Bava} for~$V_n$ 
coincide with the basic Lie algebra~\cite{DaTa}
\[
\BLb_{n-1}^{\otimes}(V_n)/ \mathcal{K}_{n-1}=\BLb_{n-1}^{\Lambda}(V_n)=(\Lambda^{n-1}L, \circ) \cong \Innder(V_n)\cong \mathfrak{so}_{n+1}.
\]

We may conclude that, although the constructions of the papers~\cite{Dzh} and~\cite{Bava} do not work in general, 
their results stay valid for the simple $n$-Lie algebra case.
In~\cite{Dzh} the finite-dimensional, irreducible representation of the 
simple $n$-Lie algebra is studied and in~\cite{Bava} irreducible 
highest weight representations of the same algebra are studied.
\end{Rem}

\section{The associated algebra construction}\label{S:Univ}

\subsection{The category of modules over an $n$-Lie algebra}
Following~\cite{Beck}, given an $n$-Lie algebra $L$ over $\kk$, we say that the category of \defn{$L$-modules} or \defn{$n$-Lie modules over $L$} is $\LMod=\Ab(\nLie\downarrow L)$, the category of abelian group objects in the comma category $(\nLie\downarrow L)$.

This definition may be unpacked as follows: an $L$-module is a $\kk$-module 
$M$ with a structure of $n$-Lie algebra on $M \oplus L$ such that $L$ 
is a subalgebra of $M \oplus L$, $M$ is an ideal of $M \oplus L$, and 
the bracket is zero if two elements are in $M$. A homomorphism 
of $L$-modules $f \colon M \to M'$ is determined by an $n$-Lie algebra 
homomorphism from $M \oplus L$ to $M' \oplus L$ which restricts to the identity on $L$. In the particular case of $n = 2$ we recover the notion of a Lie representation.

This may be further decompressed as follows. An $L$-module is a $\kk$-module 
$M$ with a linear map $[-, \dots, -] \colon (\Lambda^{n-1}L) \t M \to M$ 
satisfying the relations
\begin{multline*}
\big[x_1, \dots, x_{n-1},[y_1, \dots, y_{n-1}, m]\big] - \big[y_1, \dots, y_{n-1},[x_1, \dots, x_{n-1}, m]\big] 
\\= \sum\limits_{i = 1}^{n-1} [y_1, \dots, [x_1, \dots, x_{n-1}, y_i], \dots, y_{n-1}, m]
\end{multline*}
and 
\begin{multline*}
\big[[x_1, \dots, x_n], y_2, \dots, y_{n-1}, m\big] =\\ \sum\limits_{i=1}^{n-1} (-1)^{n-i}\big[x_1, \dots, \widehat{x}_i, \dots, x_n, [x_i, y_2, \dots, y_{n-1}, m] \big],
\end{multline*}
for all $x_i$, $y_i \in L$ and $m \in M$.

\begin{Ex}
The base ring $\kk$ is an $L$-module via the trivial action.
\end{Ex}

\subsection{The associated algebra functor}
For any $n$-Lie algebra $L$, the category $\LMod$ is an abelian variety of algebras. It is well-known that this makes it equivalent to the category of modules over the endomorphism algebra of the free $L$-module on one generator~\cite[page 106]{Freyd}. This process determines a functor $\U \colon {\nLie \to \Ass}$ from the category of $n$-Lie algebras to the category 
of associative unital $\kk$-algebras such that $\LMod$ is equivalent to the 
category $\Mod_{\U(L)}$ of ``standard'' modules over the associative algebra $\U(L)$. The following proposition gives an explicit algebraic description of the functor $\U$.

\begin{Pro}\label{Characterisation U}
Given an $n$-Lie algebra $L$, the algebra $\U(L)$ is the tensor algebra of $\Lambda^{n-1}L$ quotient by the two-sided ideal generated by
\begin{multline*}
(x_1 \wedge \cdots \wedge x_{n-1})(y_1 \wedge \cdots \wedge y_{n-1}) - (y_1 \wedge \cdots \wedge y_{n-1})(x_1 \wedge \cdots \wedge x_{n-1}) 
\\ = \sum_{i=1}^{n-1}y_1 \wedge \cdots \wedge [x_1, \dots, x_{n-1}, y_i] \wedge \cdots \wedge y_{n-1}
\end{multline*}
and
\begin{multline*}
[x_1, \dots, x_n] \wedge y_2 \wedge \cdots \wedge y_{n-1}\\ = \sum_{i=1}^{n}(-1)^{n-i}(x_1 \wedge \cdots \wedge \widehat{x}_i \wedge \cdots \wedge x_n)(x_i \wedge y_2 \wedge \cdots \wedge y_{n-1}),
\end{multline*}
for $x_i$, $y_i \in L$ and $m \in M$.\noproof
\end{Pro}

Note that when $n = 2$ we obtain the universal enveloping algebra of a Lie algebra. From the point of view of Proposition~\ref{Characterisation U}, the equivalence of categories $\LMod\simeq\Mod_{\U(L)}$ may be recovered by using that the $L$-module bracket $[x_1, \dots, x_{n-1}, m]$ defines a $\U(L)$-module action $(x_1 \wedge \cdots \wedge x_{n-1}) m$ and vice versa.

\begin{Ex}\label{E:free}
Let $L_m$ be the free $n$-Lie algebra on $m = 1$, \dots, $n-2$ generators. 
(An explicit description of the free $n$-Lie algebra can be found in~\cite{Rot}). Then $\U(L_m) = \kk$, 
since the $(n-1)$st exterior product is zero. 

Assume $m = n-1$. Then all brackets are zero and the relations of the associated algebra vanish straightforward. Hence $\U(L_{m})$ is $\kk[X]$, the commutative polynomial ring over $\kk$ with one generator.

If $m \geq n$, we can forget the elements with brackets by the second relation in Proposition~\ref{Characterisation U}. Thus we see that 
\begin{multline*}
(n-2)(x_1 \wedge \cdots \wedge x_{n-1})(y_1 \wedge \cdots \wedge y_{n-1}) + (y_1 \wedge \cdots \wedge y_{n-1})(x_1 \wedge \cdots \wedge x_{n-1}) \\
- \sum_{i=1}^{n-1}\bigg( \sum_{j=1}^{n-1} (-1)^{n-j+i} (x_1 \wedge \cdots \wedge \widehat{x}_j \wedge \cdots \wedge x_{n-1} \wedge y_i)(x_j \wedge y_1 \wedge \cdots \wedge \widehat{y}_i \wedge \cdots \wedge y_{n-1}) \bigg)
\end{multline*}
is zero.

In the case of $m = n$ this relation vanishes.
 Therefore $\U(L_n) \cong \kk \langle X_1, \dots, X_n \rangle$, 
 the non-commutative polynomial ring over $\kk$ in $n$ variables. 
 
If $m > n$ then $\U(L_n)$ is the non-commutative polynomial ring 
 on $\binom{m}{n}$ elements quotient by the two-sided ideal generated by the above relation. 
\end{Ex}

\begin{Ex}
Let $L$ be an abelian $n$-Lie algebra with $n$ generators. Then 
$\U(L) \cong \kk[X_1, \dots, X_{n-1}]$ since the first identity 
of the associated algebra makes it abelian, while the second one
vanishes.
\end{Ex}

When $n=2$ the functor $\U$ is \emph{universal} in the sense that it has a right adjoint. Let us explain why this is not possible for general $n$.

\begin{Le}\label{Lemma not Morita equivalent}
Consider $n>2$. If $F \colon \nLie \to $ $\Ass$ preserves binary sums, then there is an $n$-Lie algebra $L$ for which $F(L)$ is not Morita equivalent to $\U(L)$.
\end{Le}
\begin{proof}
Recall that if two rings (or, in particular, $\kk$-algebras) are Morita equivalent, then their centres are isomorphic.

Let $L_{1}$ be the free $n$-Lie algebra generated by one element. The coproduct of $n-1$ copies of $L_{1}$ is the free $n$-Lie algebra on $n-1$ generators, denoted by $L_{n-1}$. Its associated algebra $\U(L_{n-1})$ is $\kk[X]$ as in Example~\ref{E:free}, whose centre is itself. 

Now $\kk[X]$ cannot be Morita equivalent to $F(L_{1} + \cdots + L_{1}) \cong F(L_{1}) + \cdots + F(L_{1})$: the latter algebra being a coproduct, its centre cannot be bigger than $\kk$, so is strictly smaller than $\kk[X]$.
\end{proof}

\begin{Th}
The functor $\U \colon \nLie \to $ $\Ass$ has a right adjoint if and 
only if $n = 2$. More precisely, for $n > 2$ there is no functor $F \colon \nLie \to $ $\Ass$ 
with a right adjoint $G \colon \Ass \to \nLie$ such that there is an equivalence of categories between $\LMod$ 
and $\Mod_{\FF(L)}$ for all $L$.
\end{Th}

\begin{proof}
If $n = 2$ this result is well known. Consider the case when $n > 2$; assume that there is an adjoint pair $F \dashv G$ as required. Then, on the one hand, $F$ preserves binary sums, while on the other hand, we have an equivalence of categories $\LMod \simeq \Mod_{F(L)} \simeq \Mod_{\U(L)}$ for any $n$-Lie algebra $L$. This is in contradiction with Lemma~\ref{Lemma not Morita equivalent}.
\end{proof}

This theorem shows that for $n > 2$ there is no way we can obtain a 
functor $\FF \colon {\nLie \to\Ass}$ satisfying both requirements (U1) and (U2) of the introduction: to have an equivalence of 
categories between $\LMod$ and $\Mod_{\FF(L)}$ for all $L$ and to have a right adjoint for the functor $\FF$. In particular, it is shown that $\U(L)$ does not satisfy (U2). Of course there may still exist other functors~$\FF$ such that all $\FF(L)$ are Morita equivalent to $\U(L)$.

\begin{Rem}
If $\kk = \cc$ and $L = V_n$, then $\U(L)$ coincides with the construction of the universal enveloping algebra given in~\cite{Dzh} and~\cite{Bava}. Moreover, for any $n$-Lie algebra~$L$ such that $\BLb_{n-1}^{\Lambda}(L)$ is also a Lie algebra, $\U(L)$ is isomorphic to the universal enveloping algebra given in~\cite{Dzh}. However, if $(\Lambda^{n-1}L, \circ)$ is a Lie algebra, then $\U(L)$ might be different from the universal enveloping algebra of~\cite{Bava}.
\end{Rem}

\section{(Co)Homology theory and the associated algebra}\label{S:Hom}

Let $L$ be an $n$-Lie algebra and $M$ an $L$-module. 
Let
\[
M^L = \{m \in M \mid \text{$[x_1, \dots, x_{n-1}, m] = 0$ for all $x_i \in L$} \}
\] 
be the \defn{invariant submodule} of $M$, and let $M_L = M/LM$ be 
the \defn{coinvariant submodule}. As in Lie algebras, 
we can obtain (co)homology theories deriving the invariants and coinvariants functors.

\begin{De}
The \defn{homology groups} of $M$ \defn{with coefficients in} $L$, denoted by $\HH_\ast(L, M)$ are the left derived functors of $(-)_L$. The \defn{cohomology groups} of $M$ \defn{with coefficients in} $L$, denoted by $\HH^\ast(L, M)$ are the right derived functors of~$(-)^L$.
\end{De}

There is an immediate relation between this (co)homology theory and the associated algebra. Let $\varepsilon \colon \U(L) \to \kk$ be the $\kk$-algebra homomorphism sending the inclusion of $\Lambda^{n-1}L$ to zero. Its kernel, $\Om(L)$, is called the \defn{augmentation ideal}. Therefore, $\Om(L)$ has a $\U(L)$-module structure. 

\begin{Pro}
Let $L$ be an $n$-Lie algebra and $M$ an $L$-module. There are isomorphisms
\begin{align*}
\HH_\ast(L, M) &\cong \Tor_\ast^{\U(L)}(\kk, M), \\
\HH^\ast(L, M) &\cong \Ext^\ast_{\U(L)}(\kk, M).
\end{align*}
\end{Pro}

\begin{proof}
As in the Lie algebra case (see~\cite{Wei}), we just have to check that the underlying functors are the same.
\[
\kk \t_{\U(L)} M = \tfrac{\U(L)}{\Om(L)} \t_{\U(L)} M \cong \tfrac{M}{\Om(L) M} = \tfrac{M}{LM} = M_L,
\]
and
\[
\Hom_{\U(L)}(\kk, M) = \Hom_L(\kk, M) = M^L.\qedhere
\]
\end{proof}

Following the computations done for Lie algebras in~\cite[Section 7.4]{Wei} we obtain that $\HH_1(L, \kk) \cong \Om(L)/ \Om(L)^2$ and $\HH^1(L, \kk) \cong \Hom_\kk(\Om(L), \kk)$. In the particular case of Example~\ref{E:free}, we see that 
\begin{align*}
\HH_1(L_m, \kk) \cong \coprod_{\binom{m}{n}} \kk \qquad \text{and} \qquad \HH^1(L_m, \kk) \cong \prod_{\binom{m}{n}} \kk.
\end{align*}
These results show that the cohomology theory defined above is different from the $n$-Lie algebra cohomology theories studied in~\cite{Tak},~\cite{DaTa} and~\cite{AKMS} when $n > 2$.

\section*{Acknowledgments}

The authors would like to thank José Manuel Casas, 
Emzar Khmaladze and Manuel Ladra for us suggesting to carry out this research
and for their useful comments and remarks. The first author would like to thank 
the Institut de Recherche en Mathématique et Physique (IRMP) for its kind hospitality
during his stay in Louvain-la-Neuve.


\providecommand{\noopsort}[1]{}
\providecommand{\bysame}{\leavevmode\hbox to3em{\hrulefill}\thinspace}
\providecommand{\MR}{\relax\ifhmode\unskip\space\fi MR }
\providecommand{\MRhref}[2]{%
 \href{http://www.ams.org/mathscinet-getitem?mr=#1}{#2}
}
\providecommand{\href}[2]{#2}

\end{document}